\documentclass[reqno]{amsart}
\usepackage{amssymb}
\usepackage{amsmath}
\usepackage{amsfonts}

\newtheorem{theorem}{Theorem}
\theoremstyle{plain}

\newtheorem{corollary}{Corollary}

\newtheorem{definition}{Definition}

\newtheorem{lemma}{Lemma}

\newtheorem{proposition}{Proposition}

\numberwithin{equation}{section}
\DeclareMathOperator{\sech}{sech}

\begin{document}
\title[$f$-biharmonic $\theta _{\alpha }-$slant curves]{Unveiling $f-$%
Biharmonic $\theta _{\alpha }-$Slant Curves \\
in $\mathcal{S}-$ Space Forms}
\author{\c{S}aban G\"{u}ven\c{c}}
\address[\c{S}. G\"{u}ven\c{c}]{Balikesir University, Department of
Mathematics \\
10145, \c{C}a\u{g}\i \c{s}, Balikesir, TURKEY}
\email[\c{S}. G\"{u}ven\c{c}]{sguvenc@balikesir.edu.tr}

\begin{abstract}
In this paper, we firstly provide a concise overview of $\mathcal{S}-$%
manifolds, $f$-biharmonicity and $\theta _{\alpha }$-slant curves. We then
derive a key equation and analyze it in detail to establish the necessary
and sufficient conditions for $\theta _{\alpha }$-slant curves to be $f$%
-biharmonic. Finally, we present an example to support our findings. \medskip
\end{abstract}

\maketitle

\section{Introduction}

In 1964, J. Eells and L. Maire explored key aspects of harmonic maps, later
extending their ideas to $k$-harmonic maps \cite{EL}. Building on this, G.
Y. Jiang studied the case for $k=2$, deriving foundational variational
formulas for $2$-harmonic maps \cite{Jiang-86}. B. Y. Chen, in a
comprehensive survey, addressed biharmonic submanifolds of Euclidean space,
particularly focusing on the condition $\Delta H=0$, with $\Delta $ being
the Laplace operator and $H$ representing the mean curvature vector field 
\cite{Chen-96}. Notably, if the ambient space is Euclidean, the works of
Jiang and Chen converge.

Expanding this field, J. T. Cho, J. Inoguchi, and J. E. Lee introduced the
concept of slant curves in Sasakian manifolds, drawing a parallel to
Lancret's theorem for Euclidean space curves \cite{CIL}. Their findings
highlighted that non-geodesic curves in Sasakian $3$-manifolds are slant
curves when the ratio of $(\tau \pm 1)$ to the geodesic curvature $k$
remains constant.

Further advancements were made by D. Fetcu and C. Oniciuc, who proposed a
method to generate biharmonic submanifolds in Sasakian space forms by
utilizing the characteristic vector field flow $\xi $ \cite{Fetcu-2009-2}.
They demonstrated that this flow transforms biharmonic integral submanifolds
into biharmonic anti-invariant submanifolds. Building on this work, the
author and C. \"{O}zg\"{u}r examined biharmonic slant curves in $\mathcal{S}%
- $space forms, aiming to generalize these findings \cite{CKMS-2018}.

This paper seeks to extend the results for biharmonic slant curves by
exploring $f-$biharmonic $\theta _{\alpha }-$slant curves in $\mathcal{S}-$%
space forms. The subsequent sections provide definitions, examine the
necessary conditions for these curves to be proper $f-$biharmonic and
conclude with an example that demonstrate these concepts.

\section{Background and Definitions\label{prem}}

Let $(M,g)$ be a $(2m+s)$-dimensional Riemannian manifold. $M$ is called a 
\textit{framed metric manifold} with a \textit{framed metric structure} $%
(\phi ,\xi_{\alpha},\eta_{\alpha},g)$, $\alpha \in \{1,...,s\}$, if it
satisfies: 
\begin{gather}
\begin{array}{cccc}
\phi^{2}X=-X+\sum_{ \alpha=1}^{s}\eta_{\alpha}(X)\xi_{\alpha}, & 
\eta_{\alpha}(\phi(X))=0, & \eta_{\alpha}(\xi_{\beta})=\delta_{\alpha\beta},
& \phi(\xi_{\alpha})=0,%
\end{array}
\notag \\
g(X,Y)=g(\phi X,\phi Y)+\sum_{ \alpha=1}^{s}\eta_{\alpha}(X)\eta_{\alpha}(Y),
\label{2.2} \\
\begin{array}{cc}
\eta_{\alpha}(X)=g(X,\xi), & d\eta_{\alpha}(X,Y)=-d\eta_{\alpha}(Y,X)=g(X,%
\phi Y),%
\end{array}
\notag
\end{gather}
where $\phi$ is a $(1,1)$-type tensor field of rank $2m$; $\xi_{1},...,
\xi_{s}$ are vector fields; $\eta_{1},..., \eta_{s}$ are $1$-forms, and $g$
is a Riemannian metric on $M$; $X,Y\in TM$ and $\alpha,\beta \in \{1,...,s\}$
(see \cite{Nak-1966}, \cite{YK-1984}). The structure $(\phi, \xi_{\alpha},
\eta_{\alpha}, g)$ is said to be an $\mathcal{S}$\textit{-structure} if the
Nijenhuis tensor of $\phi$ is equal to $-2d\eta_{\alpha}\otimes \xi_{\alpha}$
for all $\alpha \in \{1,...,s\}$ \cite{Blair-1970}.

If $s=1$, a framed metric structure is equivalent to an almost contact
metric structure, and an $\mathcal{S}$-structure is equivalent to a Sasakian
structure. For an $\mathcal{S}$-structure, we have the following equations 
\cite{Blair-1970}: 
\begin{equation}
(\nabla_{X}\phi)Y=\sum_{\alpha=1}^{s}\left\{ g(\phi X,\phi
Y)\xi_{\alpha}+\eta_{\alpha}(Y)\phi^{2}X\right\},  \label{nablaf}
\end{equation}
and 
\begin{equation}
\nabla \xi_{\alpha}=-\phi,  \label{nablaxi}
\end{equation}
for all $\alpha=1,...,s$. In the case of $s=1$, (\ref{nablaxi}) can be
derived from (\ref{nablaf}).

Let $X\in T_{p}M$ be orthogonal to $\xi_{1},..., \xi_{s}$. The plane section
spanned by $\{X,\phi X\}$ is called a $\phi$\textit{-section} in $T_{p}M$,
and its sectional curvature is referred to as the $\phi$\textit{-sectional
curvature}. Let $(M, \phi, \xi_{\alpha}, \eta_{\alpha}, g)$ be an $\mathcal{S%
}$-manifold. If $M$ has constant $\phi$-sectional curvature, its curvature
tensor $R$ is given by 
\begin{equation}
\begin{array}{c}
R(X,Y)Z=\sum_{\alpha, \beta}\left\{
\eta_{\alpha}(X)\eta^{\beta}(Z)\phi^{2}Y-\eta_{\alpha}(Y)\eta^{\beta}(Z)%
\phi^{2}X\right. \\ 
\left. -g(\phi X,\phi Z)\eta_{\alpha}(Y)\xi_{\beta}+g(\phi Y,\phi
Z)\eta_{\alpha}(X)\xi_{\beta}\right\} \\ 
+\frac{c+3s}{4}\left\{ -g(\phi Y,\phi Z)\phi^{2}X+g(\phi X,\phi
Z)\phi^{2}Y\right\} \\ 
+\frac{c-s}{4}\left\{ g(X,\phi Z)\phi Y-g(Y,\phi Z)\phi X+2g(X,\phi Y)\phi
Z\right\},%
\end{array}
\label{curvaturetensor}
\end{equation}
for $X,Y,Z\in TM$ \cite{CFF}. In this case, $M$ is called an $\mathcal{S}$%
\textit{-space form} and is denoted by $M(c)$. If $s=1$, an $\mathcal{S}$%
-space form is equivalent to a Sasakian space form \cite{CB-1994}.

Let $(M,g)$ and $(N,h)$ be Riemannian manifolds, and $\phi :M\rightarrow N$
a differentiable map. A \textit{harmonic map} is a critical point of the
energy functional of $\phi$, defined as 
\begin{equation*}
E(\phi)=\frac{1}{2}\int_{M}\left\vert d\phi \right\vert^{2}\upsilon_{g},
\end{equation*}
(see \cite{ES}). Furthermore, a \textit{biharmonic map} is a critical point
of the bienergy functional 
\begin{equation*}
E_{2}(\phi)=\frac{1}{2}\int_{M}\left\vert \tau (\phi)
\right\vert^{2}\upsilon_{g},
\end{equation*}
where $\tau (\phi)=\text{trace}\nabla d\phi$ and is called \textit{the first
tension field} of $\phi$. Jiang derived the biharmonic map equation \cite%
{Jiang-86}: 
\begin{equation*}
\tau_{2}(\phi)=-J^{\phi}(\tau(\phi))=-\Delta\tau(\phi)-\text{trace}%
R^{N}(d\phi,\tau(\phi))d\phi=0,
\end{equation*}
where $J^{\phi}$ denotes the Jacobi operator of $\phi$. It is evident that
harmonic maps are biharmonic; thus, non-harmonic biharmonic maps are
referred to as \textit{proper biharmonic}.

In \cite{YeLin-2014}, Y.L. Ou proved the following lemma:

\begin{lemma}
\cite{YeLin-2014} A curve $\gamma :(a,b)\rightarrow (M,g)$ parametrized by
arclength is an $f$-biharmonic curve with a function $f:(a,b)\rightarrow
(0,\infty)$ if and only if 
\begin{equation*}
f(\nabla_{T}\nabla_{T}\nabla_{T}T-R(T,\nabla_{T}T)T)+2f^{\prime}\nabla_{T}%
\nabla_{T}T+f^{\prime\prime}\nabla_{T}T=0.
\end{equation*}
\end{lemma}

As a result, $\gamma$ is a proper $f$-biharmonic curve if and only if 
\begin{equation*}
\tau_{3}=\nabla_{T}\nabla_{T}\nabla_{T}T-R(T,\nabla_{T}T)T+2\frac{f^{\prime}%
}{f}\nabla_{T}\nabla_{T}T+\frac{f^{\prime\prime}}{f}\nabla_{T}T=0,
\end{equation*}
where $f$ is a non-constant function.

Let $\gamma :I\rightarrow M$ be a unit-speed curve in an $m$-dimensional
Riemannian manifold $(M,g)$. The curve $\gamma$ is called a \textit{Frenet
curve of osculating order} $r$ $(1\leq r\leq m)$, if there exist $g$%
-orthonormal vector fields $V_{1},V_{2},...,V_{r}$ along the curve
satisfying the Frenet equations: 
\begin{eqnarray}
T &=& V_{1} = \gamma^{\prime},  \notag \\
\nabla_{T}V_{1} &=& k_{1}V_{2},  \notag \\
\nabla_{T}V_{j} &=& -k_{j-1}V_{j-1}+k_{j}V_{j+1}, \quad 1<j<r,
\label{Frenetequations} \\
\nabla_{T}V_{r} &=& -k_{r-1}V_{r-1}.  \notag
\end{eqnarray}
Here, $k_{1},...,k_{r-1}$ are positive functions known as the curvatures of $%
\gamma$. If $k_{1}=0$, then $\gamma$ is referred to as a \textit{geodesic}.
If $k_{1}$ is a non-zero positive constant and $r=2$, $\gamma$ is called a 
\textit{circle}. If $k_{1},...,k_{r-1}$ are all non-zero positive constants,
then $\gamma$ is called a \textit{helix of order} $r$ $(r\geq 3)$. When $r=3$%
, it is simply referred to as a \textit{helix}.\bigskip

A submanifold of an $\mathcal{S}$-manifold is said to be an \textit{integral
submanifold} if $\eta_{\alpha}(X)=0$, $\alpha \in \{1,...,s\}$, where $X$ is
tangent to the submanifold \cite{KDT}. A \textit{Legendre curve} is a $1$%
-dimensional integral submanifold of an $\mathcal{S}$-manifold $(M^{2m+s},
\phi, \xi_{\alpha}, \eta_{\alpha}, g)$. More precisely, a unit-speed curve $%
\gamma :I\rightarrow M$ is a Legendre curve if $T$ is $g$-orthogonal to all $%
\xi_{\alpha}$ $(\alpha = 1,...,s)$ \cite{OG-2014}.

\section{$\protect\theta_{\protect\alpha}-$Slant Curves in $\mathcal{S}-$%
Manifolds \label{tetaalfa}}

In this section, we review the concept of $\theta_{\alpha}-$slant curves in $%
\mathcal{S}$-manifolds:

\begin{definition}
\label{tetaalfaslantcurve}\cite{Guvenc-2020} Let $M=(M^{2m+s},\phi,
\xi_{\alpha}, \eta_{\alpha}, g)$ be an $\mathcal{S}$-manifold, and let $%
\gamma: I \rightarrow M$ be a unit-speed curve. $\gamma$ is called a $%
\theta_{\alpha}-$\textit{slant curve} if there are constant angles $%
\theta_{\alpha}$ $(\alpha = 1, \dots, s)$ such that $\eta_{\alpha}(T) = \cos
\theta_{\alpha}$. These angles $\theta_{\alpha}$ are referred to as the 
\textit{contact angles} of $\gamma$.
\end{definition}

It is easy to observe that Definition \ref{tetaalfaslantcurve} generalizes
the family of slant curves to $\theta_{\alpha}-$slant curves. In particular,
a $\theta_{\alpha}-$slant curve is called \textit{slant} if all its contact
angles are equal (see \cite{CKMS-2018}).

For a $\theta_{\alpha}-$slant curve, it is known that $\eta_{\alpha}(V_2) =
0 $ for all $\alpha = 1, \dots, s.$ The following notations are used
similarly to \cite{Guvenc-2020}: 
\begin{equation*}
a = \sum_{\alpha=1}^{s} \cos^2 \theta_{\alpha}, \quad b =
\sum_{\alpha=1}^{s} \cos \theta_{\alpha}, \quad \mathcal{V} =
\sum_{\alpha=1}^{s} \cos \theta_{\alpha} \xi_{\alpha}.
\end{equation*}

\begin{corollary}
\cite{Guvenc-2020} If $\gamma$ is a slant curve, then 
\begin{equation*}
a = s \cos^2 \theta, \quad b = s \cos \theta, \quad \mathcal{V} = \cos
\theta \sum_{\alpha=1}^{s} \xi_{\alpha},
\end{equation*}
where $\theta$ denotes the common contact angle of $\gamma$.
\end{corollary}

Let $\gamma$ be a non-geodesic unit-speed $\theta_{\alpha}-$slant curve.
Then, $g(\phi T, \phi T) = 1 - a \geq 0.$ If $a = 1$, we have $\phi T = 0$,
which implies $T = \mathcal{V}$. Hence, $\nabla_T T = \nabla_{\mathcal{V}} 
\mathcal{V} = 0$, meaning that $\gamma$ is a geodesic as an integral curve
of $\mathcal{V}$.

\begin{proposition}
\cite{Guvenc-2020} For a non-geodesic unit-speed $\theta_{\alpha}-$slant
curve in an $\mathcal{S}$-manifold, 
\begin{equation*}
a = \sum_{\alpha=1}^{s} \cos^2 \theta_{\alpha} < 1.
\end{equation*}
\end{proposition}

\begin{proposition}
\cite{Guvenc-2020} For a non-geodesic unit-speed $\theta_{\alpha}-$slant
curve in an $\mathcal{S}$-manifold $(M, \phi, \xi_{\alpha}, \eta_{\alpha},
g) $, we have 
\begin{equation}
\nabla_T \phi T = (1 - a) \sum_{\alpha=1}^{s} \xi_{\alpha} + b(-T + \mathcal{%
V}) + k_1 \phi V_2.  \label{nablafT}
\end{equation}
\end{proposition}

\section{$f$-biharmonic $\protect\theta_{\protect\alpha}$-Slant Curves in $%
\mathcal{S}$-Space Forms\label{f-biharmonic}}

In this section, we consider proper $f$-biharmonic $\theta_{\alpha}$-slant
curves in $\mathcal{S}$-space forms. Let $\gamma$ be a unit-speed $%
\theta_{\alpha}$-slant curve in an $\mathcal{S}$-space form $\left(
M,\phi,\xi_{\alpha},\eta_{\alpha},g \right)$. In \cite{Guvenc-2020}, it is
shown that

\begin{equation*}
R\left( T,\nabla_{T}T \right) T = -k_{1} \left[ b^{2} + \frac{c + 3s}{4}
\left( 1 - a \right) \right] V_{2} - 3k_{1} \frac{c - s}{4} g\left( \phi
T,V_{2} \right) \phi T,
\end{equation*}

\begin{equation*}
\tau_{2}\left( \gamma \right) = \nabla_{T}\nabla_{T}\nabla_{T}T - R\left(
T,\nabla_{T}T \right) T.
\end{equation*}

Thus, we can calculate

\begin{eqnarray}
\tau_{3} &=& \nabla_{T}\nabla_{T}\nabla_{T}T - R\left( T,\nabla_{T}T \right)
T + 2\frac{f^{\prime}}{f} \nabla_{T}\nabla_{T}T + \frac{f^{\prime\prime}}{f}
\nabla_{T}T  \notag \\
&=& \left(-3k_{1}k_{1}^{\prime} - 2k_{1}^{2} \frac{f^{\prime}}{f} \right) T
\label{tau3} \\
&& + \left[ k_{1}^{\prime\prime} - k_{1}^{3} - k_{1}k_{2}^{2} + k_{1}\left[
b^{2} + \frac{c + 3s}{4}(1 - a) \right] + 2\frac{f^{\prime}}{f}%
k_{1}^{\prime} + \frac{f^{\prime\prime}}{f}k_{1} \right] V_{2}  \notag \\
&& + \left( 2k_{1}^{\prime}k_{2} + k_{1}k_{2}^{\prime} + 2\frac{f^{\prime}}{f%
} k_{1}k_{2} \right) V_{3} + k_{1}k_{2}k_{3} V_{4} + 3k_{1} \frac{c - s}{4}
g(\phi T, V_{2}) \phi T.  \notag
\end{eqnarray}

As a result, we can state the following theorem:

\begin{theorem}
\label{mainprop} $\gamma$ is a proper $f$-biharmonic $\theta_{\alpha}$-slant
curve in an $\mathcal{S}$-space form $\left(
M,\phi,\xi_{\alpha},\eta_{\alpha},g \right)$ if and only if 
\begin{equation}
3\frac{k_{1}^{\prime}}{k_{1}} + 2\frac{f^{\prime}}{f} = 0,  \label{1}
\end{equation}
\begin{equation}
k_{1}^{2} + k_{2}^{2} = \frac{k_{1}^{\prime\prime}}{k_{1}} + \frac{%
f^{\prime\prime}}{f} + 2\frac{f^{\prime}}{f} \frac{k_{1}^{\prime}}{k_{1}} +
b^{2} + \frac{c + 3s}{4}(1 - a) + 3\frac{c - s}{4} g(\phi T, V_{2})^{2},
\label{2}
\end{equation}
\begin{equation}
k_{2}^{\prime} + 2k_{2} \frac{k_{1}^{\prime}}{k_{1}} + 2k_{2} \frac{%
f^{\prime}}{f} + 3\frac{c - s}{4} g(\phi T, V_{2}) g(\phi T, V_{3}) = 0,
\label{3}
\end{equation}
\begin{equation}
k_{2}k_{3} + 3\frac{c - s}{4} g(\phi T, V_{2}) g(\phi T, V_{4}) = 0
\label{4}
\end{equation}
and $g(\tau_{3}, \phi T) = 0$.
\end{theorem}

\begin{proof}
Let $\gamma$ be a proper $f$-biharmonic $\theta_{\alpha}$-slant curve. Then $%
\tau_{3} = 0$. From equation \ref{tau3}, if we apply $T$, $V_{2}$, $V_{3}$, $%
V_{4}$, and $\phi T$, we obtain equations (\ref{1}), (\ref{2}), (\ref{3}), (%
\ref{4}), and $g(\tau_{3}, \phi T) = 0$. Conversely, if the given equations
are satisfied for a $\theta_{\alpha}$-slant curve, it can be easily shown
that $\tau_{3} = 0$. Therefore, $\gamma$ is a proper $f$-biharmonic curve.
\end{proof}

The following Lemma will be crucial for the reader to understand the results
of next cases:

\begin{lemma}
\label{ODEsol}Let $y=y(x)$ be a real valued function. Furthermore, let $%
c_{2}\geqslant 0$, $\lambda \geqslant 0$ be real constants, $\varepsilon \in
\left\{ -1,0,+1\right\} $ and $u=u(x)=2\lambda x+c_{4}$ for an arbitrary
constant $c_{4}$. Then the autonomous ODE 
\begin{equation*}
3\left( y^{\prime }\right) ^{2}-2yy^{\prime \prime }=4y^{2}\left[ \left(
1+c_{2}^{2}\right) y^{2}-\varepsilon \lambda ^{2}\right]
\end{equation*}%
has the general solution of the form%
\begin{equation}
y=\frac{\pm \sqrt{N}+M}{D},  \label{Lemma2}
\end{equation}%
where the functions $N$, $M$ and $D$ denotes

$(i)$ for $\varepsilon =+1$ and $\lambda >0:$ 
\begin{equation}
N=\lambda ^{2}\sec ^{2}u.\left[ -\left( 1+c_{2}^{2}+c_{3}^{2}\right) \sec
^{2}u+\left( 1+c_{2}^{2}-c_{3}^{2}\right) \right] ,  \label{Lemma2.i1}
\end{equation}%
\begin{equation}
M=\lambda c_{3}\sec ^{2}u,\text{ }D=\left( 1+c_{2}^{2}\right) \sec
^{2}u-\left( 1+c_{2}^{2}-c_{3}^{2}\right) ;  \label{Lemma2.i2}
\end{equation}

$(ii)$ for $\varepsilon =-1$ and $\lambda >0:$ 
\begin{equation}
N=\lambda ^{2}\sech^{2}u.\left[ \left( 1+c_{2}^{2}+c_{3}^{2}\right) \sech%
^{2}u-\left( 1+c_{2}^{2}+c_{3}^{2}\right) \right] ,  \label{Lemma2.ii1}
\end{equation}%
\begin{equation}
M=\lambda c_{3}\sech^{2}u,\text{ }D=\left( 1+c_{2}^{2}\right) \sech%
^{2}u-\left( 1+c_{2}^{2}+c_{3}^{2}\right) ;  \label{Lemma2.ii2}
\end{equation}

$(iii)$ for $\varepsilon =0$ or $\lambda =0:$%
\begin{equation}
N=0,\text{ }M=4c_{3},  \label{Lemma2.iii1}
\end{equation}%
\begin{equation}
D=c_{3}^{2}x^{2}+2c_{3}^{2}c_{4}x+c_{3}^{2}c_{4}^{2}+16c_{2}^{2}+16.
\label{Lemma2.iii2}
\end{equation}
\end{lemma}

Note that the arbitrary constants of the general solutions in Lemma \ref%
{ODEsol} are denoted by $c_{3}$ and $c_{4}$ to distinguish from $c_{1}$ and $%
c_{2},$ which will be used in some other equations later.

Now, we are all set to consider the equation $g(\tau _{3},\phi T)=0$ from
all points of view. Notice that if the coefficient of $\phi T$ vanishes,
then $g(\tau _{3},\phi T)=0$ is satisfied directly. The cases are:

\textbf{Case I. }$c = s.$

\textbf{Case II. }$c \neq s$ and $g(\phi T, V_{2}) = 0.$

When the coefficient of $\phi T$ does not vanish, we will investigate two
more cases:

\textbf{Case III. }$c \neq s$ and $\phi T \parallel V_{2}.$

\textbf{Case IV. }$c \neq s$ and $g(\phi T, V_{2}) \neq 0$ or $\pm \sqrt{1 -
a}.$

One might ask what happens if $\phi T$ itself vanishes as a vector field. In
this case, as shown in \cite{Guvenc-2020}, we have $T = \mathcal{V} =
\sum_{\alpha = 1}^{s} \cos \theta_{\alpha} \xi_{\alpha}$ and $k_{1} = 0$.
Thus, equation (\ref{1}) implies that $f$ is a constant. Consequently, $%
\gamma$ cannot be a proper $f$-biharmonic curve. Note that being proper $f$%
-biharmonic means $f$-biharmonic but not biharmonic. Likewise, being proper
biharmonic means biharmonic but not harmonic.

\textbf{Case I. }$c=s.$

In this case, equations (\ref{1}-\ref{4}) become%
\begin{equation}
3\frac{k_{1}^{\prime }}{k_{1}}+2\frac{f^{\prime }}{f}=0,  \label{1.1}
\end{equation}%
\begin{equation}
k_{1}^{2}+k_{2}^{2}=\frac{k_{1}^{\prime \prime }}{k_{1}}+\frac{f^{\prime
\prime }}{f}+2\frac{f^{\prime }}{f}\frac{k_{1}^{\prime }}{k_{1}}%
+b^{2}+s(1-a),  \label{1.2}
\end{equation}%
\begin{equation}
k_{2}^{\prime }+2k_{2}\frac{k_{1}^{\prime }}{k_{1}}+2k_{2}\frac{f^{\prime }}{%
f}=0,  \label{1.3}
\end{equation}%
\begin{equation}
k_{2}k_{3}=0,  \label{1.4}
\end{equation}%
By solving these equations, we can state the following theorem:

\begin{theorem}
\label{case1theorem} Under the assumption $c = s$, $\gamma$ is a proper $f$%
-biharmonic $\theta_{\alpha}$-slant curve in $\left(
M,\phi,\xi_{\alpha},\eta_{\alpha},g\right)$ if and only if:

$(i)$ \text{ }$\gamma $ is of osculating order $r=3$ with $%
f=c_{1}k_{1}^{-3/2}$, $\frac{k_{2}}{k_{1}}=c_{2}$, and $k_{1}=k_{1}\left(
t\right) $ is of the form 
\begin{equation*}
k_{1}\left( t\right) =\frac{\pm \sqrt{N}+M}{D},
\end{equation*}%
where the functions $N$, $M$ and $D$ are denoted by equations $(\ref%
{Lemma2.i1})$ and $(\ref{Lemma2.i2})$ and 
\begin{equation*}
\lambda _{1}=\sqrt{b^{2}+s(1-a)}>0,\text{ }u=u(t)=2\lambda _{1}t+c_{4},
\end{equation*}%
\ $c_{1}>0$, $c_{2}>0$, $c_{3}$ and $c_{4}$ are arbitrary constants, $t$ is
the arc-length parameter; or

$(ii)$ $\gamma $ is of osculating order $r=2$ and satisfies the above
equations and inequalities, except for the inequality $c_{2}>0$, where
instead $c_{2}=0$.
\end{theorem}

\begin{proof}
Let $k_{1}=k_{1}(t)$, where $t$ denotes the arc-length parameter. From
equation $(\ref{1.1})$, it is easy to see that $f=c_{1}k_{1}^{-3/2}$ for an
arbitrary constant $c_{1}>0$. Thus, we obtain%
\begin{equation}
\frac{f^{\prime }}{f}=\frac{-3}{2}\frac{k_{1}^{\prime }}{k_{1}},\quad \frac{%
f^{\prime \prime }}{f}=\frac{15}{4}\left( \frac{k_{1}^{\prime }}{k_{1}}%
\right) ^{2}-\frac{3}{2}\frac{k_{1}^{\prime \prime }}{k_{1}}.  \label{a1}
\end{equation}%
If $k_{2}=0$, then $\gamma $ is of osculating order $r=2$, and equations $(%
\ref{1.1})$ and $(\ref{1.2})$ must be satisfied. Therefore, the second
equation combined with $(\ref{a1})$ yields the ODE%
\begin{equation}
3(k_{1}^{\prime })^{2}-2k_{1}k_{1}^{\prime \prime }=4k_{1}^{2}\left[
k_{1}^{2}-\left( b^{2}+s(1-a)\right) \right] \text{.}  \label{ODE1}
\end{equation}%
This ODE can be solved using Lemma \ref{ODEsol}. Firstly, $1+c_{2}^{2}=1$
gives us $c_{2}=0.$ Since $s\geqslant 1,$ $1-a>0$ and $b^{2}\geqslant 0,$ we
have $b^{2}+s(1-a)>0.$ We can write%
\begin{eqnarray*}
b^{2}+s(1-a) &=&sgn\left( b^{2}+s(1-a)\right) .\left( \sqrt{b^{2}+s(1-a)}%
\right) ^{2} \\
&=&\varepsilon \lambda _{1}^{2},
\end{eqnarray*}%
where we denote $\lambda _{1}=\sqrt{b^{2}+s(1-a)}.$ Now that $\varepsilon
=+1 $ and $\lambda _{1}>0,$ using Lemma $\ref{ODEsol}$ $(i)$, we find that $%
k_{1}\left( t\right) $ is of the form $(\ref{Lemma2})$ where $N$, $M$ and $D$
are as in equations $(\ref{Lemma2.i1})$ and $(\ref{Lemma2.i2})$ (with $%
c_{2}=0$). Notice that $c_{2}=0$ is equivalent to $k_{2}=0=c_{2}k_{1}$ when $%
r=2$.

If $k_{2}=\text{constant}\neq 0$, we find that $f$ is a constant, indicating
that $\gamma $ is not proper $f$-biharmonic in this case. For $k_{2}\neq 
\text{constant}$, from equation (\ref{1.4}), we have $k_{3}=0$. Thus, $%
\gamma $ is of osculating order $r=3$. Using equation (\ref{a1}), equation (%
\ref{1.3}) gives us $\frac{k_{2}}{k_{1}}=c_{2}$, where $c_{2}>0$ is a
constant. Substituting these results into equation (\ref{1.2}), we obtain
the ODE%
\begin{equation*}
3(k_{1}^{\prime })^{2}-2k_{1}k_{1}^{\prime \prime
}=4k_{1}^{2}[(1+c_{2}^{2})k_{1}^{2}-\left( b^{2}+s(1-a)\right) ]
\end{equation*}%
which gives the general solution using Lemma \ref{ODEsol} and the proof is
completed.
\end{proof}

\textbf{Case II. }$c\neq s$ and $\phi T\perp V_{2}.$

In this case, $g(\phi T,V_{2})=0$. From Theorem \ref{mainprop}, we obtain 
\begin{equation*}
\begin{array}[t]{c}
3\frac{k_{1}^{\prime }}{k_{1}}+2\frac{f^{\prime }}{f}=0, \\ 
k_{1}^{2}+k_{2}^{2}=\frac{k_{1}^{\prime \prime }}{k_{1}}+\frac{f^{\prime
\prime }}{f}+2\frac{k_{1}^{\prime }}{k_{1}}\frac{f^{\prime }}{f}+b^{2}+\frac{%
c+3s}{4}(1-a), \\ 
k_{2}^{\prime }+2k_{2}\frac{f^{\prime }}{f}+2k_{2}\frac{k_{1}^{\prime }}{%
k_{1}}=0, \\ 
k_{2}k_{3}=0.%
\end{array}%
\end{equation*}%
Firstly, we need the following Lemma from \cite{Guvenc-2020}:

\begin{lemma}
\cite{Guvenc-2020} Let $\gamma $ be a $\theta _{\alpha }-$slant curve of
order $r=3$ in an $\mathcal{S}$-space form $\left( M,\phi ,\xi _{\alpha
},\eta _{\alpha },g\right) $ and $\phi T\perp V_{2}.$ Then, $\left\{
T,V_{2},V_{3},\phi T,\nabla _{T}\phi T,\xi _{1},...,\xi _{s}\right\} $ is
linearly independent. So $dimM\geqslant 5+s.$
\end{lemma}

Now we have the following theorem:

\begin{theorem}
\label{theoremcase2}Let $\gamma $ be a $\theta _{\alpha }-$slant curve in an 
$\mathcal{S}$-space form $(M^{2m+s},\phi ,\xi _{\alpha },\eta _{\alpha },g)$%
, $\alpha \in \left\{ 1,...,s\right\} ,$ $c\neq s$ and $\phi T\perp V_{2}$.
Then $\gamma $ is proper $f$-biharmonic if and only if

$(1)$ $\gamma $ is of osculating order $r=3$ with $f=c_{1}k_{1}^{-3/2},$ $%
\frac{k_{2}}{k_{1}}=c_{2},$ $m\geq 3,$%
\begin{equation*}
\left\{ T=V_{1},V_{2},V_{3},\phi T,\nabla _{T}\phi T,\xi _{1},...,\xi
_{s}\right\}
\end{equation*}%
is linearly independent and $k_{1}=k_{1}(t)$ is of the form%
\begin{equation*}
k_{1}=\frac{\pm \sqrt{N}+M}{D}
\end{equation*}%
where the functions $N$, $M$ and $D$ are denoted by

$\qquad (a)$ equations $(\ref{Lemma2.i1})$ and $(\ref{Lemma2.i2}),$ for $%
b^{2}+\left[ \left( c+3s\right) /4\right] (1-a)>0$;

\qquad $(b)$ equations $(\ref{Lemma2.ii1})$ and $(\ref{Lemma2.ii2}),$ for $%
b^{2}+\left[ \left( c+3s\right) /4\right] (1-a)<0$;

\qquad $(c)$ equations $(\ref{Lemma2.iii1})$ and $(\ref{Lemma2.iii2}),$ for $%
b^{2}+\left[ \left( c+3s\right) /4\right] (1-a)=0$;

with 
\begin{equation*}
\lambda _{2}=\sqrt{\left\vert b^{2}+\frac{\left( c+3s\right) }{4}%
(1-a)\right\vert },\text{ }u=u(t)=2\lambda _{2}t+c_{4},
\end{equation*}%
for arbitrary constants $c_{1}>0,$ $c_{2}>0,$\ $c_{3}$\ and $c_{4}$; or

$(2)$ $\gamma $ is of osculating order $r=2$ and satisfies the above
equations and inequalities, except for $m\geq 3,$ $c_{2}>0,$ 
\begin{equation*}
\text{ }\left\{ T=V_{1},V_{2},V_{3},\phi T,\nabla _{T}\phi T,\xi
_{1},...,\xi _{s}\right\} \text{ }
\end{equation*}%
is linearly independent, where instead $m\geq 2,c_{2}=0,$ 
\begin{equation*}
\text{ }\left\{ T=V_{1},V_{2},\phi T,\nabla _{T}\phi T,\xi _{1},...,\xi
_{s}\right\}
\end{equation*}
is linearly independent.
\end{theorem}

\begin{proof}
The proof is similar to the proof of Theorem \ref{case1theorem}. In this
case, we have the ODE%
\begin{equation*}
3(k_{1}^{\prime })^{2}-2k_{1}k_{1}^{\prime \prime
}=4k_{1}^{2}[(1+c_{2}^{2})k_{1}^{2}-\left[ b^{2}+\frac{\left( c+3s\right) }{4%
}(1-a)\right] .
\end{equation*}%
We use Lemma \ref{ODEsol} using the fact that%
\begin{eqnarray*}
b^{2}+\frac{c+3s}{4}(1-a) &=&sgn\left( b^{2}+\frac{c+3s}{4}(1-a)\right)
\left( \sqrt{\left\vert b^{2}+\frac{c+3s}{4}(1-a)\right\vert }\right) ^{2} \\
&=&\varepsilon \lambda _{2}^{2}.
\end{eqnarray*}%
For $b^{2}+\left[ \left( c+3s\right) /4\right] (1-a)>0,$ we have $%
\varepsilon =+1$ and $\lambda _{2}>0$. For $b^{2}+\left[ \left( c+3s\right)
/4\right] (1-a)<0,$ we find $\varepsilon =-1$ and $\lambda _{2}>0$. Finally $%
b^{2}+\left[ \left( c+3s\right) /4\right] (1-a)=0\ $gives $\varepsilon =0$
and $\lambda _{2}=0.$
\end{proof}

\textbf{Case III.} $c\neq s$, $\phi T\parallel V_{2}$.

As a result of the assumptions of this case, we have $\phi T=\epsilon \sqrt{%
1-a}V_{2}, g(\phi T,V_{2})=\epsilon \sqrt{1-a},$ $g(\phi T,V_{3})=0$ and $%
g(\phi T,V_{4})=0,$ where $\epsilon =sgn(g\left( \phi T,V_{2}\right) )=\pm
1. $ From equations $(\ref{3})$ and $\left( \ref{a1}\right) $, we can write%
\begin{equation}
k_{2}^{\prime }+2k_{2}\left( \frac{-3}{2}\frac{k_{1}^{\prime }}{k_{1}}%
\right) +2k_{2}\frac{k_{1}^{\prime }}{k_{1}}=0.  \label{3.3}
\end{equation}
Integrating $(\ref{3.3})$, we have 
\begin{equation*}
\frac{k_{2}}{k_{1}}=c_{2},
\end{equation*}%
for some constant $c_{2}>0.$ Likewise in \cite{Guvenc-2020}, for
non-constant $k_{1},$ one can show that if $\phi T=\epsilon \sqrt{1-a}V_{2},$
then%
\begin{equation*}
k_{2}=\sqrt{ad^{2}-as+b^{2}+2\epsilon bd+s},
\end{equation*}%
where $d=k_{1}/\sqrt{1-a}.$ So we get%
\begin{equation*}
\frac{\sqrt{ad^{2}-as+b^{2}+2\epsilon bd+s}}{k_{1}}=c_{2},
\end{equation*}%
which is equivalent to 
\begin{equation*}
\left( c_{2}^{2}+\frac{a}{a-1}\right) k_{1}^{2}+\frac{2\epsilon b}{a-1}%
k_{1}+\left( as-b^{2}-s\right) =0.
\end{equation*}%
If this equation is quadratic or linear in terms of $k_{1},$ then $k_{1}$
becomes a constant and $\gamma $ cannot be proper $f$-biharmonic. Let us
assume%
\begin{equation*}
c_{2}^{2}+\frac{a}{a-1}=0,
\end{equation*}%
\begin{equation*}
\frac{2\epsilon b}{a-1}=0,
\end{equation*}%
\begin{equation*}
as-b^{2}-s=0.
\end{equation*}%
The second equation above gives that $b=0.$ Then the last equation reduces
to $as-s=0,$ that is, $a=1$. But in this case $\gamma $ becomes a geodesic
as an integral curve of $\mathcal{V}$ and cannot be proper $f$-biharmonic.
Hence, we give the following result:

\begin{theorem}
\label{theoremcase3} There does not exist any proper $f$-biharmonic $\theta
_{\alpha }-$slant curve in an $\mathcal{S}$-space form $(M^{2m+s},\phi ,\xi
_{\alpha },\eta _{\alpha },g)$ with $c\neq s$ and $\phi T\parallel V_{2}$.
\end{theorem}

\textbf{Case IV. }$c\neq s$ and $g(\phi T,V_{2})$ is not $0$ or $\pm \sqrt{%
1-a}.$

In this final case, let $(M^{2m+s},\phi ,\xi _{\alpha },\eta _{\alpha },g)$
be an $\mathcal{S}$-space form and $\gamma :I\rightarrow M$ a $\theta
_{\alpha }$-slant curve of osculating order $r.$ Note that $r=2$ corresponds
to $\phi T\in $span$\left\{ V_{2}\right\} ,$ which was investigated in Case
III$.$ So let $3\leq r\leq 2m+s.$ If $\gamma $ is $f$-biharmonic, then $\phi
T\in \text{span}\{V_{2},V_{3},V_{4}\}.$ Let $\beta (t)$ denote the angle
function between $\phi T$ and $V_{2},$ that is, $g(\phi T,V_{2})=\sqrt{1-a}%
\cos \beta (t).$ If we differentiate $g(\phi T,V_{2})$ along $\gamma $ and
use equations $(\ref{Frenetequations})$ and $(\ref{nablafT})$, we get

\begin{equation}
-\sqrt{1-a}\beta^{\prime}(t)\sin \beta(t) = \nabla_{T}g(\phi T,V_{2}) =
g(\nabla_{T}\phi T,V_{2}) + g(\phi T,\nabla_{T}V_{2})  \notag
\end{equation}
\begin{equation*}
= g\left( (1-a)\sum_{\alpha=1}^{s}\xi_{\alpha} + b\left( -T + \mathcal{V}%
\right) + k_{1}\phi V_{2}, V_{2}\right) + g(\phi T,-k_{1}T+k_{2}V_{3})
\end{equation*}
\begin{equation}
= k_{2}g(\phi T,V_{3}).  \label{4.1}
\end{equation}

If we write $\phi T = g(\phi T,V_{2})V_{2} + g(\phi T,V_{3})V_{3} + g(\phi
T,V_{4})V_{4},$ Theorem \ref{mainprop} gives us

\begin{equation}
3\frac{k_{1}^{\prime }}{k_{1}}+2\frac{f^{\prime }}{f}=0,  \label{e1}
\end{equation}%
\begin{equation}
k_{1}^{2}+k_{2}^{2}=\frac{k_{1}^{\prime \prime }}{k_{1}}+\frac{f^{\prime
\prime }}{f}+2\frac{f^{\prime }}{f}\frac{k_{1}^{\prime }}{k_{1}}+\left[
b^{2}+\frac{c+3s}{4}(1-a)+3\frac{c-s}{4}(1-a)\cos ^{2}\beta \right] ,
\label{e2}
\end{equation}%
\begin{equation}
k_{2}^{\prime }+2k_{2}\frac{k_{1}^{\prime }}{k_{1}}+2k_{2}\frac{f^{\prime }}{%
f}+3\frac{c-s}{4}\sqrt{1-a}\cos \beta g(\phi T,V_{3})=0,  \label{e3}
\end{equation}%
\begin{equation}
k_{2}k_{3}+3\frac{c-s}{4}\sqrt{1-a}\cos \beta g(\phi T,V_{4})=0.  \label{e4}
\end{equation}

If we substitute $(\ref{a1})$ into $(\ref{e2})$ and $(\ref{e3})$, we find

\begin{equation}
k_{1}^{2} + k_{2}^{2} = b^{2} + \frac{c+3s}{4}(1-a) + 3\frac{c-s}{4}%
(1-a)\cos^{2}\beta - \frac{k_{1}^{\prime \prime}}{2k_{1}} + \frac{3}{4}%
\left( \frac{k_{1}^{\prime}}{k_{1}} \right)^{2},  \label{e7}
\end{equation}
\begin{equation}
k_{2}^{\prime} - \frac{k_{1}^{\prime}}{k_{1}}k_{2} + 3\frac{c-s}{4}\sqrt{1-a}%
\cos \beta g(\phi T,V_{3}) = 0.  \label{e8}
\end{equation}

If we multiply (\ref{e8}) by $2k_{2}$ and use (\ref{4.1}), we obtain

\begin{equation}
2k_{2}k_{2}^{\prime} - 2\frac{k_{1}^{\prime}}{k_{1}}k_{2}^{2} + \frac{3(c-s)%
}{4}(1-a)(-2\beta^{\prime}\cos \beta \sin \beta) = 0.  \label{e9}
\end{equation}

Let us denote $\upsilon(t) = k_{2}^{2}(t)$, where $t$ is the arc-length
parameter. Then (\ref{e9}) becomes

\begin{equation}
\upsilon^{\prime} - 2\frac{k_{1}^{\prime}}{k_{1}}\upsilon = -\frac{3(c-s)}{4}%
(1-a)(-2\beta^{\prime}\cos \beta \sin \beta),  \label{e10}
\end{equation}

which is a linear ODE. If we solve (\ref{e10}), we get the following results:

$i)$ If $\beta$ is a constant, then

\begin{equation}
\frac{k_{2}}{k_{1}}=c_{2},  \label{e12}
\end{equation}%
where $c_{2}>0$ is an arbitrary constant. From $(\ref{4.1})$ and $(\ref{bb1}%
) $, we find $g(\phi T,V_{3})=0$. Since $\left\Vert \phi T\right\Vert =\sqrt{%
1-a}$ and $\phi T=\sqrt{1-a}\cos \beta V_{2}+g(\phi T,V_{4})V_{4}$, we
obtain $g(\phi T,V_{4})=\pm \sqrt{1-a}\sin \beta .$ Using $(\ref{e2})$ and $(%
\ref{e12})$, we have

\begin{equation*}
3(k_{1}^{\prime})^{2} - 2k_{1}k_{1}^{\prime \prime} = 4k_{1}^{2}\left[%
(1+c_{2}^{2})k_{1}^{2} - b^{2} - \frac{c+3s + 3(c-s)\cos^{2}\beta}{4}(1-a)%
\right].
\end{equation*}

$ii)$ If $\beta = \beta(t)$ is a non-constant function, then

\begin{equation}
k_{2}^{2} = -\frac{3(c-s)}{4}(1-a)\cos^{2}\beta + \mu(t)k_{1}^{2},
\label{bb1}
\end{equation}

where

\begin{equation}
\mu(t) = -\frac{3(c-s)}{2}(1-a)\int \frac{\cos^{2}\beta k_{1}^{\prime}}{%
k_{1}^{3}}dt.  \label{e11}
\end{equation}

If we substitute $(\ref{bb1})$ into $(\ref{e7})$, we find

\begin{equation*}
\left[ 1+\mu(t)\right] k_{1}^{2} = b^{2} + \frac{c+3s}{4}(1-a) + \frac{3(c-s)%
}{2}(1-a)\cos^{2}\beta - \frac{k_{1}^{\prime \prime}}{2k_{1}} + \frac{3}{4}%
\left( \frac{k_{1}^{\prime}}{k_{1}} \right)^{2}.
\end{equation*}

Hence, we can state the following final theorem:

\begin{theorem}
\label{theoremcase4}Let $\gamma :I\rightarrow M$ be a $\theta _{\alpha }-$%
slant curve of osculating order $r$ in an $\mathcal{S}$-space form $%
(M^{2m+s},\phi ,\xi _{\alpha },\eta _{\alpha },g)$, where $r\geq 3$, $c\neq
s $ , $g(\phi T,V_{2})=\sqrt{1-a}\cos \beta (t)$ $\ $is not $0\ $or $\pm 
\sqrt{1-a}.$ Then $\gamma $ is proper $f$-biharmonic if and only if $%
f=c_{1}k_{1}^{-3/2}$ and

$(i)$ if $\beta $ is a constant,%
\begin{equation*}
\frac{k_{2}}{k_{1}}=c_{2},
\end{equation*}%
\begin{equation*}
3(k_{1}^{\prime })^{2}-2k_{1}k_{1}^{\prime \prime }=4k_{1}^{2}\left\{
(1+c_{2}^{2})k_{1}^{2}-\left[ b^{2}+\frac{c+3s+3(c-s)\cos ^{2}\beta }{4}(1-a)%
\right] \right\}
\end{equation*}%
\begin{equation*}
k_{2}k_{3}=\pm \frac{3(c-s)\sin 2\beta }{8}\left( 1-a\right) ,
\end{equation*}

$(ii)$ if $\beta $ is a non-constant function,%
\begin{equation*}
k_{2}^{2}=-\frac{3(c-s)}{4}(1-a)\cos ^{2}\beta +\mu (t).k_{1}^{2},
\end{equation*}%
\begin{equation*}
3(k_{1}^{\prime })^{2}-2k_{1}k_{1}^{\prime \prime }=4k_{1}^{2}\left\{ (1+\mu
(t))k_{1}^{2}-\left[ b^{2}+\frac{c+3s+3(c-s)\cos ^{2}\beta }{4}(1-a)\right]
\right\} ,
\end{equation*}%
\begin{equation*}
k_{2}k_{3}=\pm \frac{3(c-s)\sin 2\beta \sin w}{8}\left( 1-a\right) ,
\end{equation*}%
where $c_{1}$ and $c_{2}$ are positive constants, 
\begin{equation}
\phi T=\sqrt{1-a}\left( \cos \beta V_{2}\pm \sin \beta \cos wV_{3}\pm \sin
\beta \sin wV_{4}\right) ,  \label{4.f}
\end{equation}%
$w=w(t)$ is the angle function between $V_{3}$ and the orthogonal projection
of $\phi T$ onto $span\left\{ V_{3},V_{4}\right\} .$ $w$ is related to $%
\beta $ by $\cos w=\mp \beta ^{\prime }/k_{2}$ and $\mu (t)$ is given by%
\begin{equation*}
\mu (t)=-\frac{3(c-s)}{2}\left( 1-a\right) \int \frac{\cos ^{2}\beta
k_{1}^{\prime }}{k_{1}^{3}}dt.
\end{equation*}
\end{theorem}

\begin{proof}
The proof follows directly from the preceding calculations. The equation $%
\cos w = \mp \frac{\beta^{\prime}}{k_{2}}$ is derived using equations %
\eqref{4.1} and \eqref{4.f}.
\end{proof}

In case $\beta $ is a constant, we can give the following direct corollary
of Theorem \ref{theoremcase4}:

\begin{corollary}
\label{corollary}Let $\gamma :I\rightarrow M$ be a $\theta _{\alpha }-$slant
curve of osculating order $r\geq 3$ in an $\mathcal{S}$-space form $%
(M^{2m+s},\phi ,\xi _{\alpha },\eta _{\alpha },g)$, where $c\neq s$ , $%
g(\phi T,V_{2})=\sqrt{1-a}\cos \beta \ $is a constant and $\beta \in \left(
0,2\pi \right) \backslash \left\{ \frac{\pi }{2},\pi ,\frac{3\pi }{2}%
\right\} $. Then $\gamma $ is proper $f$-biharmonic if and only if $%
f=c_{1}k_{1}^{-3/2}$ , $\frac{k_{2}}{k_{1}}=c_{2}$ and $k_{1}=k_{1}(t)$ is
of the form%
\begin{equation*}
k_{1}=\frac{\pm \sqrt{N}+M}{D}
\end{equation*}%
where the functions $N$, $M$ and $D$ are denoted by

$(a)$ equations $(\ref{Lemma2.i1})$ and $(\ref{Lemma2.i2}),$ for $%
b^{2}+\left\{ \left[ c+3s+3(c-s)\cos ^{2}\beta \right] /4\right\} (1-a)>0$;

$(b)$ equations $(\ref{Lemma2.ii1})$ and $(\ref{Lemma2.ii2}),$ for $%
b^{2}+\left\{ \left[ c+3s+3(c-s)\cos ^{2}\beta \right] /4\right\} (1-a)<0$;

$(c)$ equations $(\ref{Lemma2.iii1})$ and $(\ref{Lemma2.iii2}),$ for $%
b^{2}+\left\{ \left[ c+3s+3(c-s)\cos ^{2}\beta \right] /4\right\} (1-a)=0$;

with%
\begin{equation*}
\lambda _{4}=\sqrt{\left\vert b^{2}+\frac{c+3s+3(c-s)\cos ^{2}\beta }{4}%
(1-a)\right\vert },u=u(t)=2\lambda _{4}t+c_{4},
\end{equation*}%
for arbitrary constants $c_{1}>0,$ $c_{2}>0,$\ $c_{3}$\ and $c_{4}$; 
\begin{equation*}
k_{2}k_{3}=\pm \frac{3(c-s)\sin 2\beta }{8}\left( 1-a\right)
\end{equation*}%
and $\phi T=\sqrt{1-a}\left( \cos \beta V_{2}\pm \sin \beta V_{4}\right) $.
\end{corollary}

\begin{proof}
In this case, we have the ODE%
\begin{equation*}
3(k_{1}^{\prime })^{2}-2k_{1}k_{1}^{\prime \prime
}=4k_{1}^{2}[(1+c_{2}^{2})k_{1}^{2}-\left[ b^{2}+\frac{c+3s+3(c-s)\cos
^{2}\beta }{4}(1-a)\right] .
\end{equation*}%
Using Lemma \ref{ODEsol} and Theorem \ref{theoremcase4}, the proof is clear.
\end{proof}

\section{\label{Applications}Construction of an Example in $\mathbb{R}%
^{4}(-6)$}

Firstly, let us recall the structures defined on a special $\mathcal{S}$%
-manifold. Consider $M=\mathbb{R}^{2m+s}$ with the coordinate functions $%
\{x_{1},...,x_{m},y_{1},...,y_{m},z_{1},...,z_{s}\}$ and the following
structures: 
\begin{equation*}
\xi _{\alpha }=2\frac{\partial }{\partial z_{\alpha }},\quad \alpha =1,...,s,
\end{equation*}%
\begin{equation*}
\eta ^{\alpha }=\frac{1}{2}\left( dz_{\alpha
}-\sum_{i=1}^{m}y_{i}dx_{i}\right) ,\quad \alpha =1,...,s,
\end{equation*}%
\begin{equation*}
\phi X=\sum_{i=1}^{m}Y_{i}\frac{\partial }{\partial x_{i}}%
-\sum_{i=1}^{m}X_{i}\frac{\partial }{\partial y_{i}}+\left(
\sum_{i=1}^{m}Y_{i}y_{i}\right) \left( \sum_{\alpha =1}^{s}\frac{\partial }{%
\partial z_{\alpha }}\right) ,
\end{equation*}%
\begin{equation*}
g=\sum_{\alpha =1}^{s}\eta ^{\alpha }\otimes \eta ^{\alpha }+\frac{1}{4}%
\sum_{i=1}^{m}\left( dx_{i}\otimes dx_{i}+dy_{i}\otimes dy_{i}\right) .
\end{equation*}%
Here, 
\begin{equation*}
X=\sum_{i=1}^{m}\left( X_{i}\frac{\partial }{\partial x_{i}}+Y_{i}\frac{%
\partial }{\partial y_{i}}\right) +\sum_{\alpha =1}^{s}\left( Z_{\alpha }%
\frac{\partial }{\partial z_{\alpha }}\right) \in \chi (M).
\end{equation*}%
Then, $\left( \mathbb{R}^{2m+s},\phi ,\xi _{\alpha },\eta ^{\alpha
},g\right) $ becomes an $\mathcal{S}$-space form with constant $\phi $%
-sectional curvature $-3s$. This manifold is denoted by $\mathbb{R}%
^{2n+s}(-3s)$ \cite{Hasegawa}. The following vector fields 
\begin{equation*}
X_{i}=2\frac{\partial }{\partial y_{i}},\quad X_{m+i}=\phi X_{i}=2\left( 
\frac{\partial }{\partial x_{i}}+y_{i}\sum_{\alpha =1}^{s}\frac{\partial }{%
\partial z_{\alpha }}\right) ,\quad \xi _{\alpha }=2\frac{\partial }{%
\partial z_{\alpha }}
\end{equation*}%
form a $g$-orthonormal basis of $\chi (M)$, and the Riemannian connection is
given by 
\begin{equation*}
\nabla _{X_{i}}X_{j}=\nabla _{X_{m+i}}X_{m+j}=0,\quad \nabla
_{X_{i}}X_{m+j}=\delta _{ij}\sum_{\alpha =1}^{s}\xi _{\alpha },\quad \nabla
_{X_{m+i}}X_{j}=-\delta _{ij}\sum_{\alpha =1}^{s}\xi _{\alpha },
\end{equation*}%
\begin{equation*}
\nabla _{X_{i}}\xi _{\alpha }=\nabla _{\xi _{\alpha }}X_{i}=-X_{m+i},\quad
\nabla _{X_{m+i}}\xi _{\alpha }=\nabla _{\xi _{\alpha }}X_{m+i}=X_{i},
\end{equation*}%
\cite{Hasegawa}. Let us choose $m=2$ and $s=2$. Now, let $\gamma
:I\rightarrow \mathbb{R}^{6}(-6),\gamma =\left( \gamma _{1},...,\gamma
_{6}\right) $ be a unit-speed $\theta _{\alpha }$-slant curve. We can
calculate 
\begin{equation*}
T=\frac{1}{2}\left[ \gamma _{3}^{\prime }X_{1}+\gamma _{4}^{\prime
}X_{2}+\gamma _{1}^{\prime }X_{3}+\gamma _{2}^{\prime }X_{4}+\xi _{2}\right]
,
\end{equation*}%
where we take $\theta _{1}=\frac{\pi }{2}$ and $\theta _{2}=\frac{\pi }{3}$.
Then, $g(T,T)=1$ gives us 
\begin{equation}
\left( \gamma _{1}^{\prime }\right) ^{2}+\left( \gamma _{2}^{\prime }\right)
^{2}+\left( \gamma _{3}^{\prime }\right) ^{2}+\left( \gamma _{4}^{\prime
}\right) ^{2}=3.  \label{unitspeed}
\end{equation}%
We also have 
\begin{equation}
\phi T=\frac{1}{2}\left[ -\gamma _{1}^{\prime }X_{1}-\gamma _{2}^{\prime
}X_{2}+\gamma _{3}^{\prime }X_{3}+\gamma _{4}^{\prime }X_{4}\right] .
\label{fiT}
\end{equation}%
After calculations, we obtain 
\begin{equation}
\nabla _{T}T=\frac{1}{2}\left[ 
\begin{array}{c}
\left( \gamma _{3}^{\prime \prime }+\gamma _{1}^{\prime }\right)
X_{1}+\left( \gamma _{4}^{\prime \prime }+\gamma _{2}^{\prime }\right) X_{2}
\\ 
+\left( \gamma _{1}^{\prime \prime }-\gamma _{3}^{\prime }\right)
X_{3}+\left( \gamma _{2}^{\prime \prime }-\gamma _{4}^{\prime }\right) X_{4}%
\end{array}%
\right] .  \label{nablaTT}
\end{equation}%
Let us select 
\begin{eqnarray}
k_{1} &=&\frac{1}{2+t^{2}},\quad  \label{22} \\
k_{2} &=&\frac{1}{2+t^{2}},  \notag \\
k_{3} &=&\frac{\sqrt{17}}{4}\left( 2+t^{2}\right) , \\
g(\phi T,V_{2}) &=&\frac{\sqrt{3}}{2}\cos \beta =\text{constant},  \notag \\
cos\beta &=&\pm \frac{\sqrt{2}}{6}\text{ }\left( \beta \approx 1.3329\text{
or }1.8087\right) ,  \notag \\
f &=&(2+t^{2})^{3/2},  \notag
\end{eqnarray}%
From (\ref{nablaTT}) and (\ref{22}), we get 
\begin{equation*}
\left[ 
\begin{array}{c}
\left( \gamma _{3}^{\prime \prime }+\gamma _{1}^{\prime }\right) ^{2}+\left(
\gamma _{4}^{\prime \prime }+\gamma _{2}^{\prime }\right) ^{2} \\ 
+\left( \gamma _{1}^{\prime \prime }-\gamma _{3}^{\prime }\right)
^{2}+\left( \gamma _{2}^{\prime \prime }-\gamma _{4}^{\prime }\right) ^{2}%
\end{array}%
\right] =\frac{4}{(2+t^{2})^{2}}.
\end{equation*}%
We also have $\eta ^{1}(T)=\cos \theta _{1}=0$ and $\eta ^{2}(T)=\cos \theta
_{2}=1/2$, which leads to 
\begin{equation*}
\gamma _{5}^{\prime }=\gamma _{1}^{\prime }\gamma _{3}+\gamma _{2}^{\prime
}\gamma _{4},
\end{equation*}%
\begin{equation*}
\gamma _{6}^{\prime }=1+\gamma _{1}^{\prime }\gamma _{3}+\gamma _{2}^{\prime
}\gamma _{4}.
\end{equation*}%
Using $\nabla _{T}T=k_{1}V_{2}$, we get 
\begin{eqnarray*}
g(\phi T,V_{2}) &=&\frac{1}{k_{1}}g(\phi T,\nabla _{T}T) \\
&=&\frac{2+t^{2}}{4}\left( -3+\gamma _{1}^{\prime \prime }\gamma
_{3}^{\prime }-\gamma _{1}^{\prime }\gamma _{3}^{\prime \prime }+\gamma
_{2}^{\prime \prime }\gamma _{4}^{\prime }-\gamma _{2}^{\prime }\gamma
_{4}^{\prime \prime }\right) .
\end{eqnarray*}%
Notice that for $c_{1}=1,$ $c_{2}=1,$ $c_{3}=4,$ $c_{4}=0,$ Corollary \ref%
{corollary} $(c)$ is satisfied. In fact, since $b=1/2$, $c=-6$, $s=2,$ $\sin
2\beta =\pm \sqrt{17}/9$ and $a=1/4$, we have 
\begin{equation*}
b^{2}+\frac{c+3s+3(c-s)\cos ^{2}\beta }{4}(1-a)=0,
\end{equation*}%
\begin{equation*}
k_{1}=\frac{1}{2+t^{2}}=\frac{4c_{3}}{%
c_{3}^{2}x^{2}+2c_{3}^{2}c_{4}x+c_{3}^{2}c_{4}^{2}+16c_{2}^{2}+16},
\end{equation*}%
\begin{equation*}
k_{2}=\frac{1}{2+t^{2}}=c_{2}k_{1},
\end{equation*}%
\begin{equation*}
k_{2}k_{3}=\pm \frac{\sqrt{17}}{4}=\pm \frac{3(c-s)\sin 2\beta }{8}\left(
1-a\right) .
\end{equation*}%
As a result, under these circumstances, $\gamma $ becomes a proper $f$%
-biharmonic $\theta _{\alpha }-$slant curve for $f=(2+t^{2})^{3/2}$ in $%
\mathbb{R}^{6}(-6).$

\section{Conclusions and Future Work}

In this study, the properties of $f$-biharmonic $\theta _{\alpha }-$slant
curves defined on $\mathcal{S}$-manifolds have been investigated. In the
future, it will be important to explore broader classes of these structures
and conduct studies on applications of $\theta _{\alpha }-$slant curves and
their implications in various systems. Recall that $\theta _{\alpha }-$slant
curves are defined as all the contact angles are constant seperately. As an
idea, two possible generalizations can be given as%
\begin{equation*}
a=\sum_{i=1}^{s}\cos ^{2}\theta _{\alpha }=constant
\end{equation*}%
or%
\begin{equation*}
b=\sum_{i=1}^{s}\cos \theta _{\alpha }=constant,
\end{equation*}%
where the contact angles do not need to be constant but their sum or squared
sum to be constant. Then $\theta _{\alpha }-$slant curves will be a subclass
and the results on these curves will be corollaries of those future studies.

\end{document}